\documentclass[a4paper,11pt]{article}
\usepackage[utf8]{inputenc}
\usepackage{graphicx}
\usepackage[unicode,colorlinks]{hyperref}
\usepackage{amsmath,amsthm,latexsym,amssymb,amsfonts,graphicx}
\usepackage{exscale,relsize,makeidx}
\usepackage{epstopdf}
\usepackage{epsfig}
\numberwithin{equation}{section}
\usepackage{filecontents}
\usepackage{subfig}

\newtheorem{thm}{Theorem}[section]

\newtheorem{prop}[thm]{Proposition}
\newtheorem{lem}[thm]{Lemma}
\newtheorem{rem}[thm]{Remark}

\title{Lotka-Volterra with randomly fluctuating environments: a full description}
\author{Florent Malrieu, Tran Hoa Phu}
\date{\today}

\begin{document}

\maketitle

\begin{abstract}
In this note, we study the long time behavior of Lotka-Volterra systems whose 
coefficients vary randomly. Benaïm and Lobry established that randomly 
switching between two environments that are both favorable to the same 
species may lead to four different regimes: almost sure extinction of one of 
the two species, random extinction of one species or the other and persistence 
of both species. Our purpose here is to provide a complete description of 
the model. In particular, we show that any couple of environments may lead 
to the four different behaviours of the stochastic process depending on the
jump rates.
\end{abstract}


\section{Introduction}\label{sec:intro}

For a given set of positive parameters $\varepsilon = (a,b,c,d,\alpha,\beta),$ 
consider the Lotka-Volterra differential system in $\mathbb{R}_+^2,$ is given by 
\begin{displaymath}
\left\{\begin{array}{llll}
x' = \alpha x(1-ax-by)\\
y' = \beta y (1-cx-dy)\\
(x_0,y_0) \in \mathbb{R}_+^2
\end{array}\right.
\end{displaymath}
We denote by $\displaystyle F_{\varepsilon}$ the associated vector field: 
$(x',y')=F_{\varepsilon}(x,y)$. Let us note already that when $a<c$ and $b<d$, 
the point $(1/a,0)$ attracts any path starting in $(0,+\infty)^2$. We say that 
the environment is favorable to species $x$. Similarly, when $a>c$ and $b>d$, 
the point $(0,1/d)$ attracts any path starting in $(0,+\infty)^2$. We say that the 
environment is favorable to species $y$. See \cite{MZ} for a detailed presentation 
of the four generic configurations. The environment is said to be of 
\begin{itemize}
\item Type 1: if $a<c, b<d$ (favorable to species $x$)
\item Type 2: if $a>c, b>d$ (favorable to species $y$)
\item Type 3: if $a>c, b<d$ (persistence)
\item Type 4: if $a<c, b>d$ (extinction of species $x$ or $y$ depending on the starting point)
\end{itemize}
Consider two such systems $\varepsilon_0=(a_0,b_0,c_0,d_0,\alpha_0,\beta_0)$ 
and $\varepsilon_1=(a_1,b_1,c_1,d_1,\alpha_1,\beta_1)$ and introduce the random 
process $\lbrace (X_t,Y_t,I_t)\rbrace$ on $\mathbb{R}\times \mathbb{R}\times\{0,1\}$ 
obtained by switching between these two deterministic dynamics, at 
rates $\lambda_0,\lambda_1$. More precisely, we consider the Markov process  
driven by the following generator 
\[
Lf(z,i)=F_i(z)\cdot \nabla_z f(z,i)+\lambda_i(f(z,1-i)-f(z,i)),\quad (z,i)\in \mathbb{R}^2\times\{0,1\}.
\]
Equivalently, ${(I_t)}_{t\geq 0}$ is a Markov process on $\{0,1\}$ with jump rate $\lambda_0$
and $\lambda_1$, that is 
\[
\mathbb{P}(I_{t+s} = 1-i| I_t=i,\mathcal{F}_t) = \lambda_i s + o(s),
\]
where $\mathcal{F}_t$ is the sigma field generated by $\lbrace I_u,u \leq t\rbrace$. Finally, 
$(X_t,Y_t)$ is solution of 
\[
(X_t',Y_t') = F_{\varepsilon_{I_t}}(X_t,Y_t).
\]

This process on $\mathbb{R}^2 \times \lbrace 0,1 \rbrace$ has already been 
studied in \cite{BL,MZ}. It belongs to the class of the piecewise deterministic 
Markov processes introduced by Davis \cite{MR790622}. See also \cite{MR3434260} 
for a recent review of the application areas of such processes. 
Let us introduce the invasion rates of species x and y defined in \cite{BL} as
\[
\Lambda_y = \displaystyle \int \beta_0(1-c_0x)\mu(dx,0) + \int \beta_1(1-c_1x)\mu(dx,1), 
\]
\[
\Lambda_x = \displaystyle \int \alpha_0(1-b_0y)\widehat{\mu}(dy,0) + \int \alpha_1(1-b_1x)
\widehat{\mu}(dy,1),
\]
where $\mu$ is the invariant probability measure of $(X_t,I_t)$ associated to equation: 
\[
X_t'=\alpha_{I_t}X_t(1-a_{I_t}X_t),
\]
and $ \widehat{\mu}$ is the invariant probability measure of $(Y_t,I_t)$ associated 
to equation: 
\[
Y_t'=\beta_{I_t}Y_t(1-d_{I_t}Y_t).
\]
The meaning of $\Lambda_y$ is the following: when species $y$ is close to extinction, 
species $x$ behaves approximately as $(X'_t,0) = F_{\varepsilon_{I_t}}(X_t,0)$ 
and $\Lambda_y$ is the growth rate of species $y$ with respect to invariant 
measure $\mu$ of $(X,I)$. Note that the invasion rates depend on the jump rates 
$(\lambda_0,\lambda_1) \in (0,+\infty)^2$. For every $(\lambda_0,\lambda_1)\in (0,+\infty)^2 $, 
we have two parametrizations of 
these jump rates:
\[
(s,t)\in [0,1]\times (0,+\infty):\quad st=\lambda_0,\quad(1-s)t=\lambda_1.
\]
\[
(u,v) \in [0,1]\times (0,+\infty):\quad uv=\lambda_0/{\alpha_0},\quad (1-u)v=\lambda_1/{\alpha_1}.
\]
The change of parameters $(u,v)=\xi(s,t)$ is triangular in the sense that $u$ only 
depends on $s$
\[
(u,v)=\xi(s,t)=\big( \dfrac{s\alpha_1}{(1-s)\alpha_0 + s\alpha_1},
\dfrac{t}{\alpha_0\alpha_1}((1-s)\alpha_0+s\alpha_1)\big ).
\]

Let us denote the invasion rates in the $(u,v)$ coordinates by 
\[
\tilde{\Lambda}_x(u,v) = \Lambda_x(\xi^{-1}(u,v)) \quad \text{and} \quad  
\tilde{\Lambda}_y(u,v) = \Lambda_y(\xi^{-1}(u,v)).
\]
It is established in \cite{BL} that signs of $\tilde{\Lambda}_x$ and $\tilde{\Lambda}_y$ 
determine the long time behavior of $(X_t,Y_t)$.

\bigskip

\begin{tabular}{ |p{1.3cm}||p{6cm}|p{6cm}|}
 \hline
 &$\tilde{\Lambda}_y > 0$&$\tilde{\Lambda}_y < 0$\\
 \hline
 $\tilde{\Lambda}_x >0$  &persistence of the two species &extinction of species $y$\\
 \hline
 $\tilde{\Lambda}_x <0$  &extinction of species $x$ & extinction of species $x$ or $y$\\
 \hline
\end{tabular}

\bigskip
Moreover, in \cite{BL} it is shown that two environments of Type 1 may lead to 
four regimes for the stochastic process. This surprising result is reminiscent of 
switched stable linear ODE studied in \cite{BH,BLMZ1}.
 
A fundamental property of the model is that, for all $0 \leq s \leq 1$, the vector 
field $(1-s)F_{\varepsilon_0} + sF_{\varepsilon_1}$ is the Lotka-Volterra system 
associated to the environment $\varepsilon_s = (a_s,b_s,c_s,d_s,\alpha_s,\beta_s)$ with
\begin{equation}\label{eq:alphas}
\alpha_s = s\alpha_1+ (1-s)\alpha_0 ,\quad 
a_s=\dfrac{s\alpha_1 a_1 + (1-s)\alpha_0 a_0}{\alpha_s},
\quad b_s=\dfrac{s\alpha_1 b_1 + (1-s)\alpha_0 b_0}{\alpha_s}, 
\end{equation}
\begin{equation}\label{eq:betas}
\beta_s = s\beta_1+ (1-s)\beta_0 ,\quad 
c_s=\dfrac{s\beta_1 c_1 + (1-s)\beta_0 c_0}{\beta_s},
\quad d_s=\dfrac{s\beta_1 d_1 + (1-s)\beta_0 d_0}{\beta_s}.
\end{equation}
Set
\[
I= \lbrace 0 \leq s \leq 1: a_s>c_s\rbrace \quad \text{and} 
\quad J= \lbrace 0 \leq s \leq 1: b_s > d_s\rbrace.
\]
We denote by $\tilde{I}$ the image of $I$ for the other parametrization.
\begin{rem}
As noticed in \cite{BL}, if $\varepsilon_0$ and $\varepsilon_1 $ are of Type 1 
then $I$ or $J$ may generically be empty or an open interval which 
closure is contained in $(0,1)$.
\end{rem}
Let us recall below the key result in \cite{MZ} about the expression of the invasion rates.
\begin{lem}\cite[Lemma 1.2]{MZ}\label{lem:phi}
Assume that $\varepsilon_0$ and $\varepsilon_1$ are of Type 1 and, w.l.g., $a_0<a_1$. 
The quantity $\tilde{\Lambda}_y$ can be rewritten as: 
\[
\tilde{\Lambda}_y(u,v) = \dfrac{1}{(a_1-a_0)(\dfrac{1}{\alpha_0}(1-u)
+\dfrac{1}{\alpha_1}u)}\mathbb{E}[\phi(U_{u,v})]
\]
where $\phi:[0,1] \rightarrow \mathbb{R}$ is defined by 
\[
\phi(y) =(a_0+(a_1-a_0)y)P(\dfrac{1}{a_0+(a_1-a_0)y}),
\]
where
\begin{equation}\label{eq:P}
\quad P(x)=\Big(\dfrac{\beta_1}{\alpha_1}(1-c_1x)(1-a_0x)
-\dfrac{\beta_0}{\alpha_0}(1-c_0x)(1-a_1x)\Big)\dfrac{a_1-a_0}{|a_1-a_0|},
\end{equation}
and $U_{u,v}$ is a Beta distributed Beta$(uv,(1-u)v)$ random variable.
Moreover, $\phi$ has the following properties: 
\begin{itemize}
\item If $I$ is empty then $\phi$ is nonpositive.
\item If $I$ is nonempty ($I=(u_1,u_2)$) then $\phi$ is concave, negative 
on $(0,u_1) \cup (u_2,1)$ and positive on $\tilde{I}=(u_1,u_2)$.
\end{itemize}
\end{lem}
Our first result is the precise study of the properties of $\tilde{\Lambda}_x$ and $\tilde{\Lambda}_y$ 
with two environments $\varepsilon_0, \varepsilon_1$ that are respectively of 
Type 1 and Type 2. In particular, we describe the regions where $\tilde{\Lambda}_x$ 
and $\tilde{\Lambda}_y$ are positive.

\begin{thm}\label{th:12}(Shape of the regions). Assume that $\varepsilon_0$ and $\varepsilon_1$ 
are respectively of Type 1 and Type 2. Then, there exists a function $u \mapsto v_y(u)$ 
from $(0,1) \rightarrow [0,\infty]$, such that 
$\tilde{\Lambda}_y(u,v) <0$ when $v< v_y(u)$ and $\tilde{\Lambda}_y(u,v)>0$ 
when $v>v_y(u)$. Let $a$ be the coefficient of second degree of polynomial 
$P$ given by \eqref{eq:P}.

If $a<0$, there exists $0<\alpha<\overline{\alpha}<1$ such that $v_y$ is infinite on 
$[0,\alpha]$, is decreasing and continuous on $(\alpha,\overline{\alpha})$, tends to $+ \infty$ at 
$\alpha$, tends to $0$ at $\overline{\alpha}$ and is equal to $0$ on $[\overline{\alpha},1]$. 

If $a>0$, there exists $0<\overline{\alpha}<\alpha<1$ such that $v_y$ is equal to $0$ on 
$[0,\overline{\alpha}]$, is increasing and continuous on $(\overline{\alpha},\alpha)$, 
tends to $0$ at $\overline{\alpha}$, tends to $+ \infty$ at 
$\alpha$, and is infinite on $[\alpha,1]$. 

Moreover, $\alpha$ and $\overline{\alpha}$ are explicit.
\end{thm}
 The second result is the following theorem.
\begin{thm}\label{th:99}
For any $(i,j)$ in $\lbrace 1,2,3,4 \rbrace^2$, there exist two environments 
$\varepsilon_0$ of Type i and $\varepsilon_1$ of Type j such that the associated 
stochastic process has four possible regimes depending on the jump rates.
\end{thm}

The paper is organized as follows. In Section~2 we prove the properties of 
$\tilde{\Lambda}_x$ and $\tilde{\Lambda}_y$. In Section~3 we prove 
Theorem~\ref{th:12}. In Section 4 we present illustrations obtained by numerical simulation. 
In Section 5 we study the case when the two environments are of Type 3. Finally, in Section 6, 
we prove Theorem~\ref{th:99} providing, in each case, a good couple of environments.  
\section{Expression of invasion rates}
\begin{lem} If $\varepsilon_0$ and $\varepsilon_1$ are respectively of Type 1 
and Type 2, then $\tilde{I} $ is always nonempty and  there exists $0<\alpha<1$ 
(depends on $\alpha_i,\beta_i,a_i,c_i$) such that $\tilde{I} = (\alpha,1].$
\begin{proof}
Set
\[
R=\dfrac{\beta_0 \alpha_1}{\alpha_0 \beta_1},\quad u=\dfrac{s\alpha_1}{\alpha_s},
\quad A=(a_1-a_0)(R-1),\quad B=(2a_0-c_0-a_1)R + (c_1-a_0), \quad C=(c_0-a_0)R.
\]
For any $s\in (0,1)$, we get that 
\[
c_s - a_s = \dfrac{Au^2 + Bu +C}{R(1-u)+u}
\]
where $a_s$ and $c_s$ are given by \eqref{eq:alphas} and \eqref{eq:betas}.
Set 
\[
T(u) = Au^2 + Bu + C  \quad \forall u \in [0,1].
\]
We easily get 
\[
T(0)= C = (c_0-a_0)R > 0,\quad T(1)= A+B+C=c_1-a_1<0.
\]
Because $T$ is a second degree polynomial with $T(0)>0$ and $T(1)<0$, we conclude that
\[
T(u) < 0 \Leftrightarrow u>\alpha=\dfrac{-B-\sqrt{B^2-4AC}}{2A}.
\]
Therefore $u \in \tilde{I} \Leftrightarrow T(u)<0 \Leftrightarrow u>\alpha \Leftrightarrow u \in (\alpha,1]$. 
As a consequence, $\tilde{I} = (\alpha,1]$.
\end{proof}
\end{lem}
\begin{prop} The map $\tilde{\Lambda}_y (u,v)$ satisfies the following properties: \\[5pt]
For all $u \in [0,1]$
\[
\lim_{v\rightarrow \infty} \tilde{\Lambda}_y(u,v) =\beta_u(1-\dfrac{c_u}{a_u})
\begin{cases}
>0 & \text{if } u \in \tilde{I}=(\alpha,1], \\
=0 & \text{if } u \in \partial{\tilde{I}}= \lbrace \alpha \rbrace, \\
<0 & \text{if } u \in (0,1)\setminus \overline{\tilde{I}}=[0,\alpha),
\end{cases}
\]
and
\begin{equation} \label{eq: good}
\displaystyle \lim_{v \rightarrow 0} \tilde{\Lambda}_y(u,v) 
= \dfrac{1}{\dfrac{1}{\alpha_0}(1-u)+\dfrac{1}{\alpha_1}u}  
\Bigg(\Big(\dfrac{\beta_1}{\alpha_1}(1-\dfrac{c_1}{a_1})-
\dfrac{\beta_0}{\alpha_0}(1-\dfrac{c_0}{a_0})\Big)u + 
\dfrac{\beta_0}{\alpha_0}(1-\dfrac{c_0}{a_0})\Bigg).
\end{equation}
\end{prop}
\begin{proof}
The proposition is obtained by changing variables $(s,t) \longleftrightarrow (u,v)$ 
from \cite[Prop. 2.3]{BL}.
\end{proof}
\begin{prop}There exists $0<\overline{\alpha}< 1$ such that 
$\displaystyle \lim_{v\rightarrow 0} \tilde{\Lambda}_y(u,v) > 0$ if $u > \overline{\alpha}$ 
and $\displaystyle \lim_{v \rightarrow 0}\tilde{\Lambda}_y(u,v) < 0$ if $u< \overline{\alpha}$.  
\end{prop}
\begin{proof} 
The limit in \eqref{eq: good} has the same sign than  
\[
g(u)=\Big(\dfrac{\beta_1}{\alpha_1}(1-\dfrac{c_1}{a_1})
-\dfrac{\beta_0}{\alpha_0}(1-\dfrac{c_0}{a_0})\Big)u 
+ \dfrac{\beta_0}{\alpha_0}(1-\dfrac{c_0}{a_0}), 
\quad \forall u \in [0,1].
\]
We get 
\[
g(0) = \dfrac{\beta_0}{\alpha_0}(1-\dfrac{c_0}{a_0 }) < 0
\quad \text{and}\quad 
g(1)= \dfrac{\beta_1}{\alpha_1}(1-\dfrac{c_1}{a_1}) > 0. 
\]
Since $g$ is a linear function, $\overline{\alpha}$ is the unique zero of $g$ 
and the result is clear. 
\end{proof}

\begin{prop}
Let $a$ be the coefficient of degree 2 of polynomial $P$ given by \eqref{eq:P}  
\[
a=\Big(\dfrac{\beta_1}{\alpha_1}c_1a_0-
\dfrac{\beta_0}{\alpha_0}c_0a_1\Big)\dfrac{a_1-a_0}{|a_1-a_0|}.
\]
If $a<0$ (resp. $a>0$ or $a=0$) then 
$\alpha < \overline{\alpha}$ (resp. $\alpha > \overline{\alpha}$ or $\alpha = \overline{\alpha}$).
\end{prop}

\begin{proof}
By symmetry we only consider the case $a<0$. 
Without loss of generality, we assume that $a_1 >a_0$ and $a$ becomes:
\[
a=\dfrac{\beta_1}{\alpha_1}c_1a_0-\dfrac{\beta_0}{\alpha_0}c_0a_1.
\]
To prove that $\alpha < \overline{\alpha}$, it is sufficient to prove 
$A\overline{\alpha}^2 +B\overline{\alpha} + C < 0$.
Since, by definition of $\overline{\alpha}$, 
\[
\Bigg(\dfrac{\beta_1}{\alpha_1}\Big(1-\dfrac{c_1}{a_1}\Big)-
\dfrac{\beta_0}{\alpha_0}\Big(1-\dfrac{c_0}{a_0}\Big)\Bigg)
\overline{\alpha}+\dfrac{\beta_0}{\alpha_0}\Big(1-\dfrac{c_0}{a_0}\Big) =0,
\]
we get, multiplying by $a_0a_1\alpha_1/\beta_1$, that
\begin{equation}\label{eq:alphab}
(a_0a_1-c_1a_0-Ra_1a_0+Ra_1c_0)\overline{\alpha} + Ra_1(a_0-c_0)=0. 
\end{equation}
Replacing $\overline{\alpha}$ by its expression in~\eqref{eq:alphab}, we get:
\[
 A\overline{\alpha}^2 +B\overline{\alpha} + C = 
 \dfrac{R(a_1-c_1)(a_0-a_1)(a_0-c_0)(a_0c_1-Ra_1c_0)}{(a_0a_1-a_0c_1-Ra_0a_1+Ra_1c_0)^2}.
\]
Since $c_0>a_0,a_1>c_1,a_1>a_0$ and $a_0c_1-Ra_1c_0=\dfrac{\alpha_1}{\beta_1}a <0$,  
we conclude $A\overline{\alpha}^2 + B\overline{\alpha}+C <0$.
\end{proof}

\section{Shape of the positivity region}

Recall $a=\Big(\dfrac{\beta_1}{\alpha_1}a_0c_1 -\dfrac{\beta_0}{\alpha_0}a_1c_0\Big)
\dfrac{a_1-a_0}{|a_1 - a_0|}$ is the coefficient of degree 2 of polynomial $P$ given 
by~\eqref{eq:P}.

\begin{lem}\cite[Lemma 4.1]{MZ}
Assume $\varepsilon_0$ and $\varepsilon_1$ are of Type 1. If $\tilde{I}$ is nonempty, then the map 
$(u,v) \rightarrow \mathbb{E}[\phi(U_{u,v})]$ is increasing in $v$ and concave in $u$.
\end{lem}

\begin{rem} In Benaïm and Lobry's case, if $I$ is nonempty, $\phi$ is concave and the parameter 
$a$ is always negative. In the present case, $a$ may be negative, positive or zero. Therefore, 
we have the following lemma.
\end{rem}
\begin{lem} Assume $\varepsilon_0$ and $\varepsilon_1$ are respectively of Type 1 and Type 2, 
then the shape of $\phi$ depends on the sign of $a$:
\begin{itemize}
\item If $a$ is negative, then $\phi$ is strongly concave and 
$(u,v)\rightarrow \mathbb{E}[\phi(U_{u,v})]$ is increasing in $v$ and concave in $u$.
\item If $a$ is positive, then $\phi$ is strongly convex and 
$(u,v)\rightarrow \mathbb{E}[\phi(U_{u,v})]$ is decreasing in $v$ and convex in $u$.
\item If $a$ is zero, then $\phi $ is linear and 
$(u,v)\rightarrow \mathbb{E}[\phi(U_{u_,v})]$ is constant in $v$ and linear in $u$.
\end{itemize}
\end{lem}

\begin{proof}
This is a straightforward adaptation of \cite[Lem 4.1]{MZ}.
\end{proof}

Let us conclude this section with the proof of Theorem~\ref{th:12}.

\begin{proof}[Proof of Theorem \ref{th:12}]
We consider only the case $a<0$. Set $K=(\alpha,\overline{\alpha})$. We know clearly 
that $v \rightarrow \tilde{\Lambda}_y(u,v)$ admits:
\begin{itemize}
\item negative limits at 0 and $\infty$ if $u \in [0,\alpha)$,
\item positive limits at 0 and $\infty$ if $u \in (\overline{\alpha},1]$,
\item a negative limit at 0 and a positive limit at $\infty$ if $u \in (\alpha,\overline{\alpha})$.
\end{itemize}
The fact that $v\mapsto \tilde{\Lambda}_y(u,v)$ is increasing justifies the existence 
of $v_y$, and we have
\[
\tilde{\Lambda}_y(u,v) = 0 \Leftrightarrow u \in K, v=v_y(u).
\]
Let us prove that $v_y$ is decreasing in $K$. Let $\delta_1<\delta_2$ be two points 
in $K$. Choose any $\delta_3 \in (\overline{\alpha},1)$, we get 
$\tilde{\Lambda}_y(\delta_1,v_y(\delta_1))=0$ and 
$\tilde{\Lambda}_y(\delta_3,v_y(\delta_1))>0$. Since $\tilde{\Lambda}_x(\cdot,v_y(\delta_1))$ 
is concave and $\delta_1<\delta_2<\delta_3$ we get 
$\tilde{\Lambda}_y(\delta_2,v_y(\delta_1))>0$. Since $\tilde{\Lambda}_y(\delta_2,\cdot)$ 
is increasing, we obtain $v_y(\delta_2) < v_y(\delta_1)$. 

The continuity of $v_y$ on $K$ is a straightforward consequence of the continuity of 
the function $\tilde \Lambda_y$, which is obvious from the expression~\eqref{lem:phi}. 

Let us show $v_y $ tends to $\infty$ on $\alpha$. Let 
$\lbrace u_n \rbrace \subset K: u_n \downarrow \alpha$. Since $v_y$ 
is decreasing in $K$, we get $v_y(u_n) \uparrow v \in [0,\infty]$. If $v$ 
is finite, since the zero set of $\tilde{\Lambda}_y$ is closed, by continuity, 
$\alpha \in K$ (impossible). So $v_y(u_n)\uparrow \infty$. 

Let us prove $v_y$ tends to 0 on $\overline{\alpha}.$ Let 
$\lbrace u_n \rbrace \subset K : u_n \uparrow \overline{\alpha}$. Since $v_y$ 
is decreasing in $K$, we get $v_y(u_n) \downarrow \epsilon \in [0,\infty)$. If 
$\epsilon >0$, since $u_n < \overline{\alpha}$, we obtain 
$\tilde{\Lambda}_y(u_n,\epsilon/2)<0\; \forall n$. Therefore 
$0<\tilde{\Lambda}_y(\overline{\alpha},\epsilon/2)  
=\displaystyle \lim_{n \rightarrow \infty} \tilde{\Lambda}_y(u_n,\epsilon/2) \leq 0$ 
(impossible). As a consequence, $\epsilon =0 $ and $v_y(u_n) \downarrow 0$.
\end{proof}

\section{Numerical illustrations}
 Recall that for all $u \in [0,1],\; v_y(u) \; \text{and}  \; \; v_x(u)$ are 
 the unique respective solutions of 
 \[
\tilde{\Lambda}_y(u,v) = 0 \quad \text{and} \quad \tilde{\Lambda}_x(u,v)= 0.
\]
We now consider, for a varying parameter $\rho$, the environments 
\begin{equation}\label{eq:env-rho}
 \varepsilon_0 =(1,5,2,8,3,3) \quad \text{and} \quad \varepsilon_1 =(2,11,1,\rho,2,1.8).
\end{equation}
\begin{figure}[!ht]
\begin{center}
\includegraphics[scale=0.7]{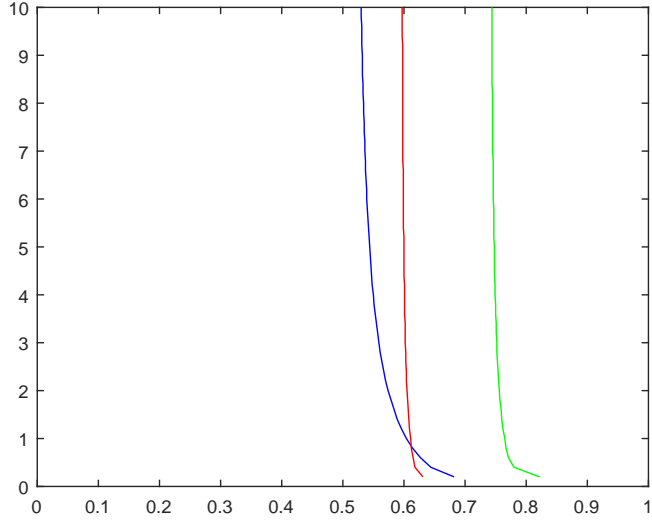}
\caption{The blue curve is the graph of $v_y$ (it does not depend on $\rho$); 
the green and red curves are $v_x$ for the environments given in~\eqref{eq:env-rho} 
with $\rho=10$ and $\rho=9$ respectively.}
\label{fig:critical}
\end{center}
\end{figure}
Figure $\ref{fig:critical}$ represents the "critical" functions $v_y$ and $v_x$ 
for different choices of the environments. Thanks to \cite{BL}, these 
plots give us information about how many regimes we can observe when the jump 
rates are modified. For example, the plot for $\rho = 10$ has three regimes: 
extinction of $x$ (on the right of the green curve), persistence (between the green and blue curves) 
and extinction of $y$ (on the left of the blue curve). For $\rho=9$, there is an additional zone
(above the red curve and below the blue curve) that corresponds to jump rates leading to 
random extinction of a species.

\section{Switching between two persistent Lotka-Volterra systems}

Let us assume that $\varepsilon_0$ and  $\varepsilon_1$ are of Type 3. In this case, one can 
easily get that extinction of species $y$ is not possible if $u$ is to close of $0$ or $1$; in other
words, $[0,1]\setminus \tilde{I}$ is either empty or is an open interval which closure 
is contained in $[0,1]$. Recall
\[
R=\dfrac{\beta_0 \alpha_1}{\alpha_0 \beta_1},\quad A=(a_1-a_0)(R-1),\quad B=(2a_0-c_0-a_1)R + (c_1-a_0), \quad C=(c_0-a_0)R.
\]
Then, we get that 
\begin{align*}
[0,1]\setminus \tilde{I} \neq \emptyset \Leftrightarrow \begin{cases}
A < 0 \\
\Delta = B^2 -4AC > 0\\
0< \dfrac{-B-\sqrt{\Delta}}{2A}<1.
\end{cases}
\end{align*}
Moreover, if $[0,1] \setminus \tilde{I} $ is nonempty, then the map 
$(u,v)\rightarrow \mathbb{E}[\phi(U_{u,v})]$ is (strictly) decreasing in $v$ and convex in $u$.
This is a straightforward adaptation of Lemma 4.1 in \cite{MZ}.

Figure~\ref{fig:3-3} provides the shape of $v_x$ and $v_y$ for the environments
$\varepsilon_0=(6,1,4,2,1,5)$ and $\varepsilon_1 = (3,3,2,5.5,5,1)$. Once again, 
the switched process has four regimes depending on the jump rates.

 \begin{rem}
 We see a surprising result : although both vector fields are persistent, the stochastic process may 
 lead to the extinction of one of the two species. 
 \end{rem} 

\begin{figure}[!ht]
\includegraphics[scale=0.8]{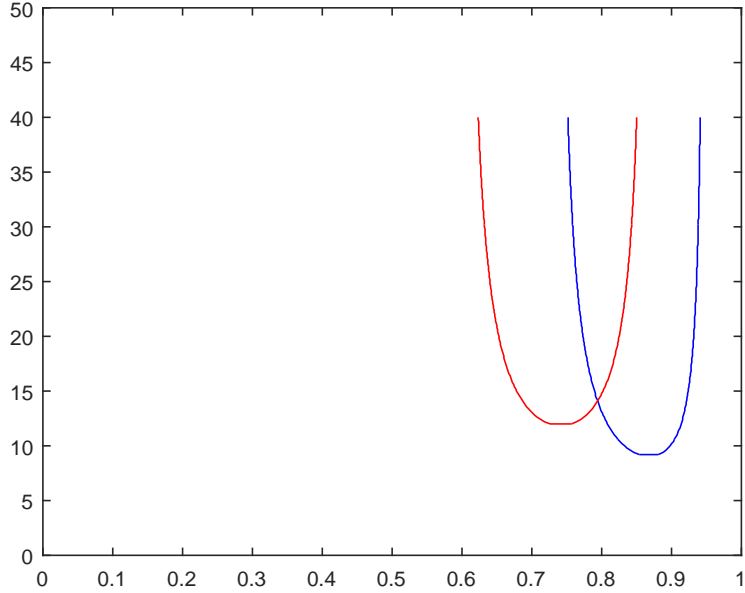}
\caption{Graph of $v_y$ (blue curve) and $v_x$ (red curve) for the 
environments $\varepsilon_0=(6,1,4,2,1,5)$ and $\varepsilon_1 = (3,3,2,5.5,5,1)$.}
\label{fig:3-3}
\end{figure}

\section{General case: proof of Theorem \ref{th:99}} 

The following array presents, for any couple of types, an example of two 
environments that are associated to a stochastic process with four regimes 
depending on the jump rates. The first line has been obtained in \cite{BL}. 
The second line is studied in Section 2. The fifth line is studied in Section 5. 
The reader can easily check that the other cases correspond to 
Figure \ref{fig:contour}.


\begin{center}
\begin{tabular}{ |c||c|c|c|c|c|c||c|c|c|c|c|c|}
 \hline
 $(F_0,F_1)$&$a_0$&$b_0$&$c_0$&$d_0$&$\alpha_0$&$\beta_0$&$a_1$&$b_1$&$c_1$&$d_1$&$\alpha_1$&$\beta_1$\\
 \hline
 \hline
  Type 1-1  &1&1&2&2&1&5&3&3&4&3.5&5&1\\
 \hline
 Type 1-2  &1&5&2&8&3&3&2&11&1&9&2&1.8\\
 \hline
 Type 1-3 &1&1&3.5&2&1&5&5&3&4&5.5&5&1\\
\hline
 Type 1-4 &1&1&2&3.5&1&5&3&4&4&3&5&1\\
 \hline
 Type 3-3 &6&1&4&2&1&5&3&3&2&5.5&5&1\\
 \hline
 Type 3-4 &6&1&4&8&1&5&3&10&4&7&5&1\\
 \hline
 Type 4-4 &2&2&1&1&5&1&7&3.5&4&3&1&5\\
 \hline
\end{tabular}
\end{center}

\begin{figure}[htp]
  \centering
  \subfloat[Type 1-3]{\label{fig:edge-a}\includegraphics[scale=0.5]{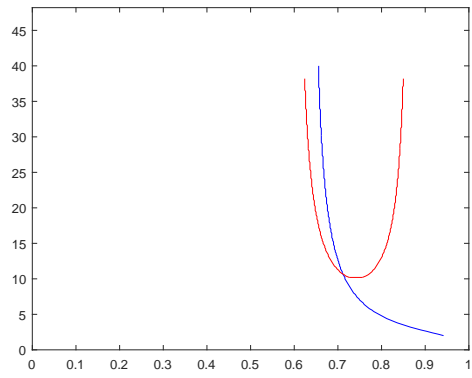}}
  \hspace{5pt}
   \subfloat[Type 1-4]{\label{fig:contour-b}\includegraphics[scale=0.5]{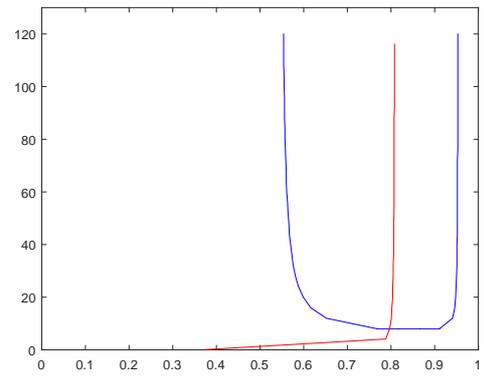}}
  \hspace{5pt}
  \subfloat[Type 3-4]{\label{fig:contour-c}\includegraphics[scale=0.5]{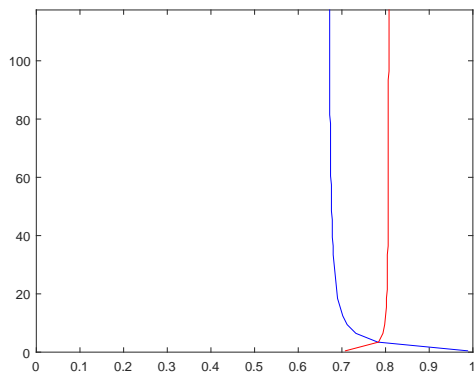}}
  \hspace{5pt}
  \subfloat[Type 4-4]{\label{fig:contour-d}\includegraphics[scale=0.5]{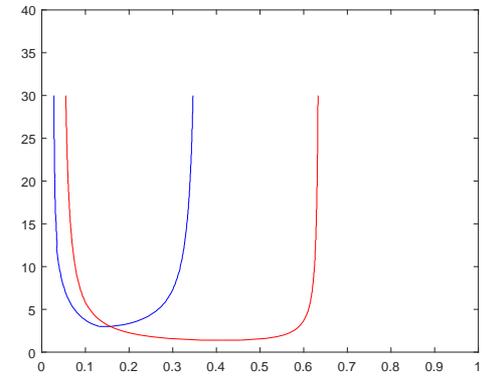}}
  \caption{Graph of $v_y$ (blue curve) and $v_x$ (red curve) 
   for the four last cases.}
  \label{fig:contour}
\end{figure}

\paragraph*{Acknowledgements.}
This work has been written during the stay of Tran Hoa Phu in Tours for his intership 
of the French-Vietnam Master in Mathematics. We acknowledge financial support 
from the French ANR project ANR-12-JS01-0006-PIECE.

\bibliography{Malrieu-Phu}
\bibliographystyle{amsplain}

{\footnotesize

\noindent Florent \textsc{Malrieu}, 
e-mail: \texttt{florent.malrieu(AT)univ-tours.fr}

\medskip

\noindent\textsc{Laboratoire de Math\'ematiques et Physique Th\'eorique 
(UMR CNRS 7350), F\'ed\'eration Denis Poisson (FR CNRS 2964), Universit\'e 
Fran\c cois-Rabelais, Parc de Grandmont, 37200 Tours, France.}

\bigskip

   \noindent Tran Hoa~\textsc{Phu},
e-mail: \texttt{thphu1(AT)yahoo.com}

 \medskip

 \noindent\textsc{Laboratoire de Math\'ematiques et Physique Th\'eorique 
(UMR CNRS 7350), F\'ed\'eration Denis Poisson (FR CNRS 2964), Universit\'e 
Fran\c cois-Rabelais, Parc de Grandmont, 37200 Tours, France.}
}

\end{document}